\documentclass[conference]{IEEEtran}
\ifCLASSINFOpdf
\else
  \usepackage[dvips]{graphicx}
\fi

\newtheorem{theorem}{\bf Theorem}[section]
\newtheorem{lemma}{\bf Lemma}[section]
\newtheorem{remark}{\bf Remark}[section]
\newtheorem{corollary}{\bf Corollary}[section]
\newtheorem{assumption}{\bf Assumption}[section]

\usepackage{amsmath,amssymb}
\usepackage{algorithm}
\usepackage{algorithmic}
\usepackage{graphicx}
\usepackage{tikz}
\usetikzlibrary{shapes,arrows}

\begin{document}
\IEEEoverridecommandlockouts
\title{Robust Adaptive Dynamic Programming for Optimal Nonlinear Control Design}

\author{\IEEEauthorblockN{Yu Jiang and Zhong-Ping Jiang}
\IEEEauthorblockA{Department of Electrical and
Computer Engineering\\
Polytechnic Institute of New York University,
Brooklyn, NY 11201\\
Email: yjiang06@students.poly.edu, zjiang@poly.edu}\thanks{This work is supported in part by the U.S. National Science Foundation, under grants DMS-0906659 and ECCS-1101401.}}
\maketitle

\begin{abstract}
This paper studies the robust optimal control design for uncertain nonlinear systems from a perspective of robust adaptive dynamic programming (robust-ADP). The objective is to fill up a gap in the past literature of ADP where dynamic uncertainties or unmodeled dynamics are not addressed. A key strategy is to integrate tools from modern nonlinear control theory, such as the robust redesign and the backstepping techniques as well as the nonlinear small-gain theorem, with the theory of ADP. The proposed robust-ADP methodology can be viewed as a natural extension of ADP to uncertain nonlinear systems. A practical learning algorithm is developed in this paper, and has been applied to a sensorimotor control problem.
\end{abstract}

\IEEEpeerreviewmaketitle

\section{Introduction}
Reinforcement learning (RL) \cite{sutton1998reinforcement} is an important branch in machine learning theory. It is concerned with how an agent should modify its actions based on the reward from its reactive unknown environment so as to achieve a long term goal. In 1968, Werbos pointed out that the policy iteration technique devised in \cite{howard} for dynamic programming can be employed to perform RL \cite{werbos1968}. Starting from then, many real-time RL methods for finding online optimal control policies have emerged and they are broadly called approximate/adaptive dynamic programming (ADP) \cite{werbos1972beyond}, \cite{werbos1992approximate} or neurodynamic programming \cite{bertsekas}. See
\cite{Abu-Khalaf2008}, \cite{Al-Tamimi2007model},
\cite{jiang2011auto},
\cite{lewis2011reinforcement},
\cite{murray},
\cite{vamvoudakis2011},
\cite{vrabie2009nn},
\cite{vrabie2009adaptive},
for some recently developed results.

In the past literature of ADP, it is commonly assumed that the system order is known and the state variables are either fully available or reconstructible from the output; see  \cite{lewis2007}, \cite{lewis2011reinforcement} and reference therein. However, in practice, the system order may be unknown due to the presence of dynamic uncertainties (or unmodeled dynamics), which are motivated by engineering applications in situations where the exact mathematical model of a physical system is not easy to be obtained. Of course, dynamic uncertainties also make sense for the mathematical modeling in other branches of science such as biology and economics. This problem, often formulated as robust control, cannot be viewed as a special case of output feedback control, and the ADP methods developed in the past literature may not only fail to guarantee optimality, but also the stability of the closed-loop system when dynamic uncertainty occurs.
To fill up the above-mentioned gap in the past literature of ADP, we recently developed a new theory of robust adaptive dynamic programming (robust-ADP) \cite{jiang2011tnnls}, \cite{jiang2011cdc}, \cite{jiang2012bookchapter}, which can be viewed as a natural extension of ADP to linear and partially linear systems with dynamic uncertainties.

The primary objective of this paper is to study robust-ADP designs for genuinely  nonlinear systems in the presence of dynamic uncertainties.
We first decompose the open-loop system into two parts: The {\it system model} (ideal environment) with known system order and  fully accessible state, and the {\it dynamic uncertainty}, with unknown system order and dynamics, interacting with the ideal environment.
In order  to handle the dynamic interaction between two systems, we then resort to the gain assignment idea \cite{jiang1997}, \cite{jiang1994small}, \cite{pralywang1996}. More specifically, we need
to assign a suitable gain
for the system model with disturbance in the sense of Sontag's
input-to-state stability (ISS)  \cite{sontagbook}.
The backstepping, robust redesign, and small-gain techniques
in modern nonlinear control theory are incorporated into the robust-ADP theory, such that
the system model is made ISS with an arbitrarily small gain.
At last, the nonlinear small-gain theorem \cite{jiang1994small} is applied to analyze the stability for the interconnected systems.


Throughout this paper,
vertical bars $|\!\cdot\!|$ represent the Euclidean norm for vectors, or the induced matrix norm for matrices.
 For any piecewise continuous function $u$, $\|u\|$ denotes ${\rm sup}\{|u(t)|,t\ge 0\}$.
A function $\gamma:\mathbb{R}_{+}\rightarrow\mathbb{R}_{+}$ is said to be of class $\mathcal{K}$ if it is continuous, strictly increasing  with $\gamma(0)=0$. It is of class $\mathcal{K}_\infty$ if additionally $\gamma(s)\rightarrow\infty$ as $s\rightarrow\infty$.
A function $\beta:\mathbb{R}_{+}\times\mathbb{R}_+\rightarrow\mathbb{R}_+$ is of class $\mathcal{KL}$ if $\beta(\cdot,t)$ is of class $\mathcal{K}$ for every fixed $t\ge 0$, and $\beta(s,t)\rightarrow 0$ as $t\rightarrow\infty$  for each fixed $s \geq 0$.
%
The notation $\gamma_1>\gamma_2$ means $\gamma_1(s)>\gamma_2(s)$, $\forall s>0$.

\section{Preliminaries}

In this section, let us review a policy iteration technique to solve optimal control problems \cite{saridis1979}.

To begin with, consider the system
\begin{eqnarray}
\dot x=f(x)+g(x)u \label{eq:affine}
\end{eqnarray}
where $x\in\mathbb{R}^{n}$ is the system state, $u\in\mathbb{R}$ is the control input, $f,g:\mathbb{R}^{n}\rightarrow \mathbb{R}^{n}$ are locally Lipschitz functions.
For any initial condition $x_0\in \mathbb{R}^n$, the cost function associated with (\ref{eq:affine}) is defined as
\begin{eqnarray}
J(x_0)=\int_0^{\infty}\left[Q(x)+ru^2\right]dt,~~x(0)=x_0 \label{eq:cost0}
\end{eqnarray}
where $Q(\cdot)$ is a positive definite function,
and $r>0$ is a constant.
In addition,  assume there exists an {\it admissible} control policy $u=u_0(x)$ in the sense that, under this policy, the system \eqref{eq:affine} is globally asymptotically stable and the cost \eqref{eq:cost0} is finite.
By \cite{lewis1995optimal}, the control policy that minimizes the cost (\ref{eq:cost0}) can be solved from the following Hamilton-Jacobi-Bellman (HJB) equation:
\begin{eqnarray}
\label{eq:HJB0}
0=\nabla V(x) f(x)+Q(x)-\frac{1}{4r}\left[\nabla V(x)g(x)\right]^2
\end{eqnarray}
with the boundary condition $V(0)=0$. Indeed, if the solution $V^*(x)$ of  \eqref{eq:HJB0} exists, the optimal control policy is given by
\begin{eqnarray}
\label{eq:HJBu}
u^*(x)=-\frac{1}{2r}g(x)^T\nabla V^{*}(x)^T.
\end{eqnarray}

In general, the analytical solution of \eqref{eq:HJB0} is difficult to be solved. However, if $V^*(x)$ exists, it can be approximated using the policy iteration technique \cite{saridis1979}:

\begin{enumerate}
\item Find an admissible control policy $u_0(x)$.
\item For any integer $i\ge 0$, solve for $V_i(x)$, with $V_i(0)=0$, using
\begin{eqnarray}
0=\nabla V_i(x)\left[f(x)\!+\!g(x)u_i(x)\right]\!+\!Q(x)\!+\!ru_i(x)^2. \label{eq:pi1}
\end{eqnarray}

\item Update the control policy using
\begin{eqnarray}
u_{i+1}(x)=-\frac{1}{2r}g(x)^T\nabla V_i(x)^T. \label{eq:pi2}
\end{eqnarray}
\end{enumerate}

Convergence of the policy iteration (\ref{eq:pi1}) and (\ref{eq:pi2}) is concluded in the following theorem, which can be seen as a trivial extension of Theorem 4 in \cite{saridis1979}.

\begin{theorem}
Consider $V_i(x)$ and $u_i(x)$ defined in (\ref{eq:pi1}) and (\ref{eq:pi2}). Then, for all $i=0,1,\cdots$,
\begin{eqnarray}
V_{i+1}(x)\le V_i(x),~~\forall x\in R^n
\end{eqnarray}
and $u_i(x)$ is admissible.
In addition, if the solution $V^*(x)$ of \eqref{eq:HJB0} exists, then for each fixed $x$, $V_i(x)$ and $u_i(x)$ converge pointwise  to $V^*(x)$ and $u^*(x)$, respectively.
\end{theorem}

\section{Online Learning via Robust-ADP}

In this section, we develop the robust-ADP methodology for nonlinear systems as follows:
\begin{eqnarray}
\dot{w}&=&\Delta_w(w,x) \label{eq:mc1}\\
\dot{x}&=&f(x)+g(x)\left[u+\Delta(w,x)\right] \label{eq:mc2}
\end{eqnarray}
where $x\in\mathbb{R}^{n}$ is the measured component of the state available for feedback control,  $w\in\mathbb{R}^{p}$ is the unmeasurable part of the state with unknown order $p$, $u\in\mathbb{R}$ is the control input, $\Delta_w:\mathbb{R}^{p}\times\mathbb{R}^{n}\rightarrow\mathbb{R}^{p}$,  $\Delta:\mathbb{R}^{p}\times \mathbb{R}^n\rightarrow \mathbb{R}$ are unknown locally Lipschitz functions, $f$ and $g$ are defined the same as in \eqref{eq:affine} but are assumed to be unknown.


Our design objective is to find online the control policy which stabilizes the system at the origin. Also, in the absence of the dynamic uncertainty (i.e., $\Delta=0$ and the $w$-subsystem is absent), the control policy becomes the optimal control policy that minimizes \eqref{eq:cost0}.

\subsection{Online policy iteration}

The iterative technique introduced in Section 2 relies on the knowledge of $f(x)$ and $g(x)$. To remove this requirement, we develop a novel online policy iteration technique, which can be viewed as the nonlinear extension of \cite{jiang2011auto}.

To begin with, notice that (\ref{eq:mc2}) can be rewritten as
\begin{eqnarray}
\dot x&=&f(x)+g(x)u_i(x)+g(x)v_i
\label{eq:new_form}
\end{eqnarray}
where $v_i=u+\Delta-u_i$. For each $i\ge 0$, the time derivative of $V_i(x)$ along the solutions of (\ref{eq:new_form}) satisfies
\begin{eqnarray}
\dot{V}_i(x)
&=&-Q(x)-{r}u_i^2(x)-2ru_{i+1}(x)v_i.\label{eq:derivative}
\end{eqnarray}

Integrating both sides of \eqref{eq:derivative} on any time interval $[t,t+T]$, it follows that
\begin{eqnarray}
&&V_i(x(t+T))-V_i(x(t))\nonumber\\
&=&\int_{t}^{t+T}\left[-Q(x)-ru_i^2(x)-2r u_{i+1}(x)v_i\right]\!dt.\label{eq:onlineformular0}
\end{eqnarray}

Notice that if $u_i(x)$ is given, the unknown functions $V_i(x)$ and $u_{i+1}(x)$ can be approximated using \eqref{eq:onlineformular0}. To be more specific, for any given compact set $\Omega\subset\mathbb{R}^n$ containing the origin as an interior point, let $\{\phi_j(x)\}_{j=1}^{\infty}$ be an infinite sequence of linearly independent smooth basis functions on $\Omega$, where $\phi_j(0)=0$ for all $j=1,2,\cdots$.
%
%
%
%
%
Then, for each $i=0,1,\cdots$, the cost function and the control policy are approximated by
$\hat{V}_i(x)=\sum\limits_{j=1}^{N_1}\hat{c}_{i,j}\phi_j(x)$, and
$\hat{u}_{i+1}(x)=\sum\limits_{j=1}^{N_2}\hat{w}_{i,j}\phi_j(x)$, respectively,
where $N_1>0$, $N_2>0$ are two sufficiently large integers, and $\hat{c}_{i,j}$, $\hat{w}_{i,j}$ are constant weights to be determined.

Replacing $V_i(x)$, $u_i(x)$, and $u_{i+1}(x)$ in \eqref{eq:onlineformular0} with their approximations, we obtain
\begin{eqnarray}
&&\sum\limits_{j=1}^{N_1}\hat{c}_{i,j}\left[\phi_j(x(t_{k+1}))-\phi_j(x(t_{k}))\right]\nonumber\\
&=&-\int_{t_k}^{t_{k+1}}2r\sum\limits_{j=1}^{N_2}\hat{w}_{i,j}\phi_j(x)\hat{v}_idt\label{eq:onlineformular}\\
&&-\int_{t_k}^{t_{k+1}}\left[Q(x)+r\hat{u}_i^2(x)\right]dt+e_{i,k}\nonumber
\end{eqnarray}
where $\hat{u}_0=u_0$, $\hat{v}_i=u+\Delta-\hat{u}_i$, and $\{t_k\}_{k=0}^l$ is a strictly increasing sequence with $l>0$ a sufficiently large integer. Then, the weights $\hat{c}_{i,j}$ and $\hat{w}_{i,j}$ can be solved in the sense of least-squares (i.e., by minimizing $\sum_{k=0}^le_{i,k}^2$).

Now, starting from $u_0(x)$, two sequences $\{\hat{V}_i(x)\}_{i=0}^{\infty}$, and $\{\hat{u}_{i+1}(x)\}^{\infty}_{i=0}$
can be generated via the online policy iteration technique \eqref{eq:onlineformular}. Next, we show the convergence of the sequences to $V_i(x)$ and $u_{i+1}(x)$, respectively.


%
%


\begin{assumption}
\label{as:PE}
There exist $l_0>0$ and $\delta>0$, such that for all $l\ge l_0$, we have
\begin{eqnarray}
\frac{1}{l}\sum\limits_{k=0}^{l}\theta_{i,k}^T\theta_{i,k}\ge \delta I_{N_1+N_2}
\end{eqnarray}
where
\begin{eqnarray*}
\theta_{i,k}^{T}=\left[\begin{array}{c}
\phi_{1}(x({t_{k+1}}) )-\phi_{1}(x(t_k) )\\
\phi_{2}(x({t_{k+1}}) )-\phi_{2}(x(t_k) )\\
 \vdots \\
\phi_{N_{1}}(x({t_{k+1}}) )-\phi_{N_{1}}(x(t_k) )\\
 2r\int_{t_{k}}^{t_{k+1}}\phi_{1}(x )\hat{v}_i(x )dt\\
  2r\int_{t_{k}}^{t_{k+1}}\phi_{2}(x )\hat{v}_i(x )dt\\
   \vdots\\
   2r \int_{t_{k}}^{t_{k+1}}\phi_{N_{2}}(x )\hat{v}_i(x )dt\end{array}\right]\in\mathbb{R}^{N_1+N_2}.
\end{eqnarray*}
\end{assumption}

\begin{assumption}
\label{as:inv}
For all $t\ge0$, we have $x(t)\in\Omega$.
\end{assumption}

Notice that, Assumption \ref{as:inv} is not very restrictive and can be satisfied if  $\Omega$ is an invariant set for the $x$-subsystem. This issue will be further elaborated in Section \ref{sec:imp}.

\begin{theorem}
Under Assumptions \ref{as:PE} and \ref{as:inv}, for each $i\ge 0$, we have
\begin{eqnarray}
\lim_{N_{1},N_{2}\rightarrow\infty}\hat{V}_i(x)~~&=&{V}_{i}(x),\\
\lim_{N_{1},N_{2}\rightarrow\infty}\hat{u}_{i+1}(x)&=&{u}_{i+1}(x),
\end{eqnarray}
for all $x\in\Omega.$
\begin{proof}
See the Appendix.
\end{proof}
\end{theorem}

\begin{corollary}
\label{cor:con}
Under Assumptions \ref{as:PE} and \ref{as:inv}, for any arbitrary $\epsilon>0$, there exist integers $i^*>0$, $N_1^*>0$ and $N^*_2>0$, such that
\begin{eqnarray*}
|\hat{V}_i(x)-V^*(x)|\le\epsilon,~~{\rm and}~~
|\hat{u}_{i+1}(x)-u^*(x)|\le\epsilon,
\end{eqnarray*}
for all $x\in\Omega$, if $i>i^*$, $N_1>N_1^*$, and $N_2>N_2^*$.
\end{corollary}

\subsection{Robust redesign}
In the presence of the dynamic uncertainty, we redesign the approximated optimal control policy so as to achieve asymptotic stability. This method is an integration of optimal control theory \cite{lewis1995optimal} with the gain assignment technique \cite{jiang1994small}, \cite{pralywang1996}.
To begin with, let us assume the following:


\begin{assumption}
\label{as:HJB}
There exists a function ${\underline{\alpha}}$ of class $\mathcal{K}_{\infty}$, such that for $i=0,1,\cdots$,
\begin{eqnarray}
{\underline{\alpha}}(|x|)\le V_i (x),~~\forall x\in\mathbb{R}^n.
\end{eqnarray}
In addition,
assume there exists a constant $\epsilon>0$ such that $Q(x)-\epsilon^2|x|^2$ is a positive definite function.
\end{assumption}

Notice that, we can also find a class $\mathcal{K}_{\infty}$ function ${\bar{\alpha}}$, such that for $i=0,1,\cdots$,
\begin{eqnarray}
V_i (x)\le{\bar{\alpha}}(|x|),~~\forall x\in\mathbb{R}^n.
\end{eqnarray}

\begin{assumption}
\label{as:duc}
Consider (\ref{eq:mc1}). There exist functions ${\underline{\lambda}}, {\bar{\lambda}}\in\mathcal{K}_{\infty}$, $\kappa_1,\kappa_2,\kappa_3\in\mathcal{K}$, and  positive definite functions $W$ and $\kappa_4$,  such that for all $w\in\mathbb{R}^{p}$ and $x\in\mathbb{R}^{n}$, we have
\begin{eqnarray}
&{{\underline{\lambda}}}(|w|)\le W(w)\le{{\bar{\lambda}}}(|w|),&\\
&|\Delta(w,x)|\le \max\{\kappa_1(|w|),\kappa_2(|x|)\},&
\end{eqnarray}
together with the following implication:
\begin{equation}
W(w)\ge \kappa_3 (|x|)\Rightarrow \nabla W(w)\Delta_w(w,x)\le-\kappa_4(w).
\end{equation}
\end{assumption}
Assumption \ref{as:duc} implies that the $w$-system \eqref{eq:mc1} is input-to-state stable (ISS) \cite{sontagbook} when $x$ is considered as the input, i.e.,
\begin{eqnarray}
|w(t)|\le \beta_w(|w(0)|,t)+\gamma_w(\|x\|)
\end{eqnarray}
where $\beta_w$ is of class $\mathcal{KL}$ and $\gamma_w$ is of class $\mathcal{K}$.

%
%

Now, consider the following type of control policy
\begin{eqnarray}
u_{ro}(x)=\left[1+\frac{r}{2}\rho^{2}(|x|^2)\right]\hat{u}_{i^*+1}(x) \label{eq:acon}\label{eq:acon1}
\end{eqnarray}
where $i^*>0$ is a sufficiently large integer as defined in Corollary \ref{cor:con}, $\rho$ is a smooth, non decreasing function, with $\rho(s)>0$ for all $s\ge 0$. Notice that $u_{ro}$ can be viewed as a robust redesign of the approximated optimal control law
$\hat{u}_{i^*+1}$.

As in \cite{jiang1997}, let us define a class $\mathcal{K}_\infty$ function $\gamma$ by
\begin{eqnarray}
\label{eq:gammafun}
\gamma(s)=\frac{1}{2}\epsilon\rho(s^2)s,~~\forall s\ge 0.
\end{eqnarray}

In addition, define
\begin{eqnarray}
e_{ro}(x)&=&\frac{r}{2}\rho^{2}(|x|^2)\left[\hat{u}_{i^*+1}(x)-{u}_{i^*+1}(x)\right]\nonumber\\
&&+\hat{u}_{i^*+1}(x)-u_{i^*}(x).
\end{eqnarray}

%

\begin{theorem}
Under Assumptions \ref{as:HJB} and \ref{as:duc}, suppose
\begin{eqnarray}
{\gamma}>\max\{\kappa_2,\kappa_1\circ{\underline{\lambda}}^{-1}\circ\kappa_3\circ{\underline{\alpha}}^{-1}\circ{\bar{\alpha}}\},
\label{eq:sg_cd}
\end{eqnarray}
and the following implication holds for some constant $d>0$:
\begin{eqnarray}
0<V_{i^*}(x)\le  d \Rightarrow |e_{ro}(x)|<{\gamma}(|x|).
\label{eq:sg_cd_add}
\end{eqnarray}
Then, the closed-loop system composed of (\ref{eq:mc1}), (\ref{eq:mc2}), and (\ref{eq:acon1}) is asymptotically stable at the origin. In addition, there exists $\sigma\in\mathcal{K}_{\infty}$, such that
$\Omega_{i^*}=\{(w,x): \max\left[\sigma(V_{i^*}(x)),W(w)\right]\le \sigma(d)\}$
is an estimate of the region of attraction of the closed-loop system.
\end{theorem}
\begin{proof}
See the Appendix.
\end{proof}

\begin{remark}
In the absence of the dynamic uncertainty (i.e., $\Delta=0$ and the $w$-system is absent), the control policy \eqref{eq:acon} can be replaced by $\hat{u}_{i^*+1}(x)$, which is an approximation of the optimal control policy $u^*(x)$ that minimizes the following cost function
\begin{eqnarray}
J(x_0)=\int_0^{\infty}\left[Q(x)+{r}u^2\right]dt,~~x(0)=x_0. \label{eq:cost}
\end{eqnarray}
\end{remark}


\section{Robust-ADP with unmatched dynamic uncertainty}

In this section, we extend the robust-ADP methodology to nonlinear systems with unmatched dynamic uncertainties.
To begin with, consider the system:
\begin{eqnarray}
\dot w &=& \Delta_w(w,x)                                     \label{eq:muc1}       \\
\dot{x} & = & f(x)+g(x)\left[z+\Delta(w,x)\right]        \label{eq:muc2}       \\
\dot{z} & = & f_{1}(x,z)+u+\Delta_{1}(w,x,z)             \label{eq:muc3}
\end{eqnarray}
where $[x^T,z]^T\in\mathbb{R}^{n}\times \mathbb{R}$ is the measured component of the state available for feedback control; $w$, $u$, $\Delta_w$, $f$, $g$, and $\Delta$ are defined in the same way as in \eqref{eq:mc1}-\eqref{eq:mc2}; $f_1:\mathbb{R}^{n}\times \mathbb{R}\rightarrow \mathbb{R}$  and $\Delta_1:\mathbb{R}^{p}\times \mathbb{R}^n\times \mathbb{R}\rightarrow \mathbb{R}$ are locally Lipschitz functions and are assumed to be unknown.
\begin{assumption}
\label{as:d1}
There exist class $\mathcal{K}$ functions $\kappa_5,\kappa_6,\kappa_7$, such that the following inequality holds:
\begin{eqnarray}
|\Delta_1(w,x,z)|\le\max\{\kappa_{5}\left(|w|\right),\kappa_{6}\left(|x|\right),\kappa_{7}\left(|z|\right)\}.
\end{eqnarray}
\end{assumption}


\subsection{Online learning}
Let us define a virtual control policy $\xi=u_{ro}$, as defined in \eqref{eq:acon}. Then, a state transformation can be performed as $\zeta=z-\xi$. Along the trajectories of \eqref{eq:muc2}-\eqref{eq:muc3}, it follows that
\begin{eqnarray}
\dot{\zeta}
=\bar f_{1}(x,z)+u+\Delta_{1}-\bar g_1(x)\Delta
\end{eqnarray}
where
$\bar f_1(x,z)=f_{1}(x,z)-\frac{\partial\xi}{\partial x}f(x)-\frac{\partial\xi}{\partial x}g(x)z$, and
$\bar g_1(x)=\frac{\partial \xi}{\partial x}g(x)$
are two unknown functions that can be approximated by
$\hat{f}_1(x,z)=\sum_{j=1}^{N_3}\hat{w}_{f,j}\psi_j(x,z)$ and
$\hat{g}_1(x)=\sum_{j=0}^{N_4-1}\hat{w}_{g,j}\phi_j(x)$, respectively,
where $\{\psi_j(x,z)\}_{j=1}^{\infty}$ is a sequence of linearly independent basis functions on some compact set $\Omega_1\in\mathbb{R}^{n+1}$ containing the origin as its interior, $\phi_0(x)\equiv 1$, $\hat{w}_{f,j}$ and $\hat{w}_{g,j}$ are constant weights to be trained. As in the matched case, $\Omega_1$ is selected to be an invariant set for the system \eqref{eq:muc2} and \eqref{eq:muc3}.

\subsubsection{Phase-one learning}

To approximate the virtual control input $\xi$ for the $x$-subsystem, the same procedure as in \eqref{eq:onlineformular} can be applied,  with $\hat{v}_i=z+\Delta-\hat{u}_i$.

\subsubsection{Phase two learning}
To approximate the unknown functions $\bar{f}_1$ and $\bar{g}_1$,
The constant weights can be solved, in the sense of least-squares, from
\begin{eqnarray}
%
&&\frac{1}{2}\zeta^2(t_{k+1}')-\frac{1}{2}\zeta^2(t_{k}')\nonumber\\
&=&\int_{t_k'}^{t_{k+1}'}\left[\sum_{j=1}^{N_3}\hat{w}_{f,j}\psi_j(x,z)-\sum_{j=0}^{N_4-1}\hat{w}_{g,j}\phi_j(x)\Delta\right] \zeta dt\nonumber\\
&&+\int_{t_k'}^{t_{k+1}'}(u+\Delta_1)\zeta dt+\bar{e}_{k}\label{eq:eq2}
\end{eqnarray}
where $\{t_{k}'\}_{k=1}^{l}$ is a strictly increasing positive constant sequence with $l>0$ a sufficiently large integer, and $\bar{e}_k$ denotes the approximation error. Similarly as in the previous section, let us introduce the following assumption:
\begin{assumption}
\label{as:PE1}
There exist $l_1>0$ and $\delta_1>0$, such that for all $l\ge l_1$, we have
\begin{eqnarray}
\frac{1}{l}\sum\limits_{k=0}^{l}\bar{\theta}_{k}^T\bar{\theta}_{k}\ge \delta_1 I_{N_3+N_4}
\end{eqnarray}
where
\begin{eqnarray*}
\bar{\theta}_{k}^T=\left[\begin{array}{c}
 \int_{t_{k}'}^{t_{k+1}'}\psi_{1}(x,z )\zeta dt\\
  \int_{t_{k}'}^{t_{k+1}'}\psi_{2}(x,z )\zeta dt\\
   \vdots\\
   \int_{t_{k}'}^{t_{k+1}'}\psi_{N_{3}}(x,z )\zeta dt\\
 \int_{t_{k}'}^{t_{k+1}'}\phi_{0}(x )\Delta\zeta dt\\
  \int_{t_{k}'}^{t_{k+1}'}\phi_{1}(x )\Delta\zeta dt\\
   \vdots\\
   \int_{t_{k}'}^{t_{k+1}'}\phi_{N_{4}-1}(x )\Delta\zeta dt\end{array}\right]\in\mathbb{R}^{N_3+N_4}.
\end{eqnarray*}
\end{assumption}

\begin{theorem}
\label{thm:con2}
Consider $(x(0),z(0))\in\Omega_1$. Then, under Assumption \ref{as:PE1} we have
\begin{eqnarray}
\lim_{N_{3},N_{4}\rightarrow\infty}\hat{f}(x,z)&=&\bar{f}_{1}(x,z),\\
\lim_{N_{3},N_{4}\rightarrow\infty}\hat{g}(x)~~~&=&\bar{g}_{1}(x),~~\forall (x,z)\in\Omega_1.
\end{eqnarray}
\end{theorem}
Theorem \ref{thm:con2} can be proved following the same idea as in the proof of Theorem 3.1, and is omitted here for want of space.

\subsection{Robust redesign}
Next, we study robust stabilization of the system \eqref{eq:muc1}-\eqref{eq:muc3}. To this end, let $\kappa_8$ be a function of $\mathcal{K}$ such that
\begin{eqnarray}
\kappa_8(|x|)\ge |\xi(x)|,~~\forall x\in\mathbb{R}^n.
\end{eqnarray}
Then, Assumption \ref{as:d1} implies
\begin{eqnarray*}
|\Delta_1|
&\le&\max\{\kappa_{5}\left(|w|\right),\kappa_{6}\left(|x|\right),\kappa_{7}\left(|z|\right)\}\\
&\le&\max\{\kappa_{5}\left(|w|\right),\kappa_{6}\left(|x|\right),\kappa_{7}\left(|\xi|+\kappa_8(|x|)\right)\}\\
&\le&\max\{\kappa_{5}\left(|w|\right),\kappa_{9}\left(|X_1|\right)\}
\end{eqnarray*}
where $\kappa_9 (s) = \max\{\kappa_6,\kappa_7\circ\kappa_8\circ(2s),\kappa_7\circ(2s)\}$,
$\forall s \geq 0$. In addition, we denote $\tilde\kappa_1=\max\{\kappa_1,\kappa_5\}$, $\tilde\kappa_2=\max\{\kappa_2,\kappa_9\}$,  $\gamma_1(s)=\frac{1}{2}\epsilon \rho(\frac{1}{2}s^2)s$, and
\begin{eqnarray}
U_{i^*}(X_1)=V_{i^*}(x)+\frac{1}{2}\zeta^{2}.
\end{eqnarray}
%
%
Notice that, under Assumptions \ref{as:HJB} and \ref{as:duc}, there exist $\bar\alpha_1,\underline\alpha_1\in \mathcal{K}_{\infty}$, such that
$\underline\alpha_1(|X_1|)\le U_{i^*}(X_1)\le \bar\alpha_1(|X_1|)$.

The control policy can be approximated by
\begin{eqnarray}
\!\!{u}_{ro1}
\!\!&\!\!=\!\!&\!\!-\hat{f}_1(x,z)
+2r\hat{u}_{i^*+1}(x)\nonumber\\
&&-\frac{\hat{g}^2(x)\rho_1^2(|X_1|^2)\zeta}{4}-\epsilon^2\zeta\label{eq:uapp2}\\
&&-\frac{\rho_{1}^2(|X_1|^2)\zeta}{4}-\frac{\epsilon^2\rho^2(\zeta^2)\zeta}{2\rho^2(|x|^2)} \nonumber
\end{eqnarray}
where ${X}_1=[x^T,\zeta]^T$, and $\rho_{1}(s)=2\rho(\frac{1}{2}s)$.

Next, define the approximation error as
\begin{eqnarray}
\!\!e_{ro1}(X_1)
\!\!&\!\!=\!\!&\!\!-\bar{f}_1(x,z)+\hat{f}_1(x,z)\nonumber\\
&&+2r\left[{u}_{i^*+1}(x)-\hat{u}_{i^*+1}(x)\right]\nonumber\\
&&-\frac{\left[\bar{g}_1^2(x)-\hat{g}_1^2(x)\right]\rho_1^2(|X_1|^2) \zeta}{4}
\end{eqnarray}
Then, conditions for asymptotic stability are summarized in the following Theorem:

\begin{theorem}
Under Assumptions \ref{as:HJB}, \ref{as:duc}, and \ref{as:d1}, if
\begin{eqnarray}
\label{eq:sg2}
\gamma_1>\max\{\tilde\kappa_2,\tilde\kappa_1\circ\underline\lambda^{-1}\circ\kappa_3\circ\underline\alpha_1^{-1}\circ\bar\alpha_1\},
\end{eqnarray}
and if the following implication holds for some constant $d_1>0$:
\begin{eqnarray*}
0<U_{i^*}(X_1)\le  d_1\Rightarrow \max\{|e_{ro1}(X_1)|,|e_{ro}(x)|\}<{\gamma}_1(|X_1|),
\label{eq:sg_cd_add}
\end{eqnarray*}
then the closed-loop system comprised of \eqref{eq:muc1}-\eqref{eq:muc3}, and \eqref{eq:uapp2} is asymptotically stable at the origin. In addition, there exists $\sigma_1\in\mathcal{K}_{\infty}$, such that
\begin{eqnarray*}
{\Omega}_{1,i^*}=\{(w,X_1): \max\left[\sigma_1(U_{i^*}(X_1)),W(w)\right]\le \sigma_1(d_1)\}
\end{eqnarray*}
is an estimate of the region of attraction.
\end{theorem}
\begin{proof}
See the Appendix.
\end{proof}

\begin{remark}
In the absence of the dynamic uncertainty (i.e., $\Delta=0$, $\Delta_1=0$ and the $w$-system is absent), the smooth functions $\rho$ and $\rho_{1}$ can all be replaced by $0$, and the system dynamics becomes
\begin{eqnarray}
\dot X_1=F_1(X_1)+G_1u_{o1}
\end{eqnarray}
where $F_1(X_1)=\left[\begin{array}{c}f(x)+g(x)\zeta+g(x)\xi\\-\nabla V_{i^*}(x)g(x)\end{array}\right]$, $G_1=\left[\begin{array}{c}0\\1\end{array}\right]$, and $u_{o1}=-\epsilon^2\zeta^2$. As a result, it can be concluded that the control policy $u=u_{o1}$ is an approximate optimal control policy with respect to the cost function
\begin{eqnarray}
J_1(X_1(0))=\int_{0}^{\infty}\left[Q_{1}\left(x,\zeta\right)+\frac{1}{2\epsilon^2}u^{2}\right]dt
\end{eqnarray}
with $X_1(0)=[x_0^T,z_0-u_{i^*}(x_0)]^T$ and
$Q_{1}\left(x,\zeta\right)=Q\left(x\right)+\frac{1}{4r}\left[\nabla V_{i^*}(x)g(x)\right]^{2}+\frac{\epsilon^2}{2}\zeta^{2}.$
\end{remark}

\section{Implementation Issues}
\label{sec:imp}
In this section, we study a few implementation issues on the robust-ADP based online learning methodology, and give a practical algorithm. Due to the space limitation, we will mainly focus on the systems with matched dynamic uncertainties. These results can be easily extended to the unmatched case.

%
%
%
%

\subsection{The compact set for approximation}
\begin{assumption}
\label{as:incn}
The closed-loop system composed of \eqref{eq:mc1}, \eqref{eq:mc2}, and
\begin{eqnarray}
u=u_0(x)+e
\label{eq:initial_controller}
\end{eqnarray}
is ISS when $e$, the exploration noise, is considered to be the input.
\end{assumption}

The reason for imposing Assumption \ref{as:incn} is two fold. First, like in many other policy iteration based ADP algorithms, an initial admissible control policy is desired. In this paper we further assume the initial control policy is stabilizing in the presence of dynamic uncertainties. Such an assumption is feasible and realistic by means of the designs in \cite{jiang1997}, \cite{pralywang1996}. Second, by adding the exploration noise, we are able to satisfy Assumptions \ref{as:PE} and \ref{as:PE1}, and at the same time keep the system solutions bounded.

Under Assumption \ref{as:incn}, we can find a compact set $\Omega_0$ which is an invariant set of the closed-loop system compose of \eqref{eq:mc1}, \eqref{eq:mc2}, and $u=u_0(x)$. In addition, we can also let $\Omega_0$ contain $\Omega_{i^*}$ as its subset. Then, the compact set for approximation can be selected as $\Omega=\{x|\exists w, {\rm~s.t.}~ (x,w)\in \Omega_0 \}$.

\subsection{Two-loop optimization scheme}
In general cases, it may be difficult to determine the number of basis functions to be used for approximation. In this paper we propose a two-loop online optimization scheme as shown in Fig. \ref{fig:twoloop}. In the inner loop, least-squares method is used to train the weights. If the residual sum of errors is greater than a given threshold $\bar{\epsilon}>0$, in the outer loop the number of basis functions are increased and online data are recollected  to solve the minimization problem until sufficient small residual error can be obtained.
\begin{figure}
\centering
\includegraphics[scale=.3]{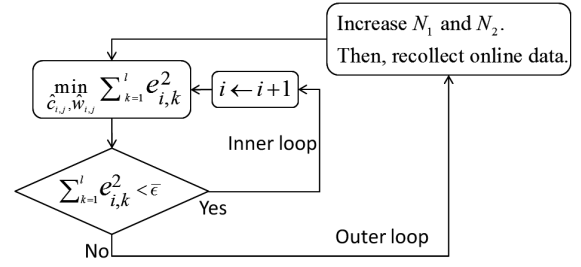}
\caption{Two-loop online optimization scheme}
\label{fig:twoloop}
\end{figure}
\subsection{Robust-ADP algorithm}

The robust-ADP algorithm is given in Algorithm \ref{alg:CAOC}.
\begin{algorithm}
\begin{enumerate}
\item[1.] Let $(w(0),x(0))\in\Omega_{i^*}$, employ the initial control policy \eqref{eq:initial_controller} and collect the system state and input information.
\item[2.] Apply the online policy iteration using \eqref{eq:onlineformular}, and redesign the control policy using \eqref{eq:acon}.
\item[3.] Terminate the exploration noise $e$.
\item[4.] If $(w(t),x(t))\in\Omega_{i^*}$, apply the approximate robust optimal control policy  \eqref{eq:acon}.
\end{enumerate}
\caption{Robust-ADP Algorithm}
\label{alg:CAOC}
\end{algorithm}

\section{Application to a single-joint human arm movement control problem}
In this section, we apply the proposed online learning strategy to study a sensorimotor control problem. A linear version of this problem has been studied in \cite{jiang-spmb}.

Consider a single-joint arm movement as shown in Fig. \ref{fig:sj}, where the position of the elbow is fixed. The dynamic model is shown below \cite{shadbook}.
\begin{eqnarray}
I\ddot{\theta}=-mgl\cos(\theta)+n+T_m
\end{eqnarray}
where $m$ is the mass of segment, $I$ is the inertia, $g$ is the gravitational constant, $l$ is the distance of the center of mass from the joint, $\theta$ is the joint angular position, $T_m$ is the input to the muscle from the motorneurons, and $n$ denotes the inputs from the neural integrator, which can be modeled by a low pass filter as follows with a time constant $\tau_N$.
\begin{eqnarray}
\dot n=-\frac{n}{\tau_N}+T_m.
\end{eqnarray}

Let us define $x_1=\theta-\theta_0$, $x_2=\dot\theta$,
$w=n-\frac{\tau_Nmgl\cos(\theta_0)}{\tau_N+1}-Ix_2$,
$u=T_m-\frac{mgl\cos(\theta_0)}{\tau_N+1}$,
 where $\theta_0$ is the desired end point angular position.  Then, the system can be converted to
\begin{eqnarray}
\dot w_{~}&=&-\frac{1+\tau_N}{\tau_N}(w+Ix_2)\\
&&-{2mgl}\sin(\frac{x_1}{2})\sin(\frac{x_1}{2}+\theta_0)\\
\dot x_1&=&x_2\\
\dot x_2&=&\frac{2mgl}{I}\sin(\frac{x_1}{2})\sin(\frac{x_1}{2}+\theta_0)\\
&&+\frac{1}{I}\left(u+Ix_2+w\right)\nonumber
\end{eqnarray}

To apply the proposed robust-ADP method, the basis functions we used are polynomials with degrees less than or equal to five. The invariant set is chosen to contain the region $\{(w,x_1,x_2):|w|\le 1, |x_1|\le 0.8, |x_2|\le 3.5\}$. Only for simulation purpose, we set $\theta_0=\frac{\pi}{4}$, $m=1.65$, $l=0.179$, $g=9.81$, $I=0.0779$. An initial control policy is set to be $u_0=-0.5x_1-0.5x_2$. The initial condition is set to be $w(0)=1$, $x_1(0)=-\frac{\pi}{4}$, and $x_2(0)=0$. The optimal cost is specified as $J=\int_0^{\infty}\left(100x_1^2+x_2^2+u^2\right)dt$.

In this simulation, the convergence is attained after $10$ iterations. It can be seen from Fig. \ref{fig:cost} that the approximated cost function $\hat{V}_{10}(x)$ is remarkably reduced compared with the initial approximated cost $\hat{V}_{0}(x)$. Also, in Fig. \ref{fig:speed}, we compare the speed curves under the initial control policy, the policy after two iterations, and the policy after 10 iterations. Clearly, after enough iteration steps, the speed profile becomes a bell-shaped curve which is consistent with experimental observations (see, for example, \cite{atkeson}).
\begin{figure}
\centering
\includegraphics[scale=0.3]{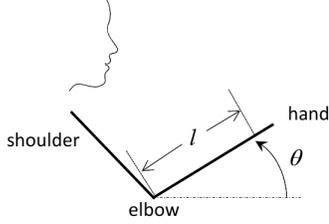}
\caption{Single-joint arm movement control problem.}
\label{fig:sj}
\end{figure}
\begin{figure}
\centering
\includegraphics[scale=0.5]{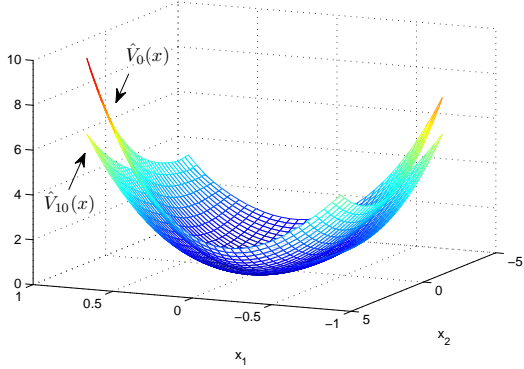}
\caption{Comparison of the approximated cost functions.}
\label{fig:cost}
\end{figure}
\begin{figure}
\centering
\includegraphics[scale=0.5]{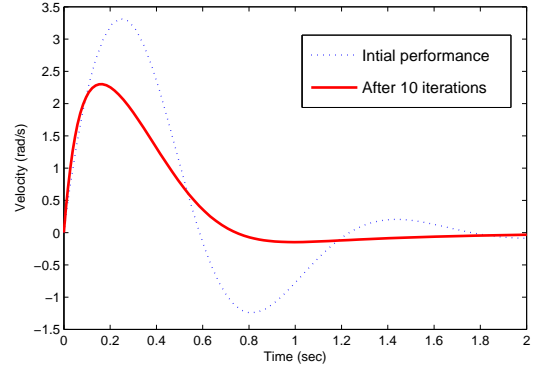}
\caption{Comparison of the speed profiles.}
\label{fig:speed}
\end{figure}

\section{Conclusions}
In this paper, computational robust optimal controller design has been studied for nonlinear systems with dynamic uncertainties.
Both the matched and the unmatched cases are studied.
We have presented for the first time a recursive, online, adaptive optimal controller design
when dynamic uncertainties, characterized by input-to-state stable systems with unknown order and states/dynamics, are taken into consideration. We have achieved this goal by integration of approximate/adaptive dynamic programming (ADP) theory and several tools recently developed within the nonlinear control community.
Systematic robust-ADP based online learning algorithm has been developed.
Rigorous stability analysis based on Lyapunov and small-gain techniques is carried out.
The effectiveness of the proposed methodology has been validated by its application to a single-joint arm movement control problem.

\appendix
\section*{Proof of Theorem 3.1}
To begin with, given $\hat{u}_i$, let $\tilde V_i(x)$ be the solution of the following equation with $\tilde V_i(0)=0$.
\begin{eqnarray}
\nabla{\tilde{V}_{i}}(x)\left(f(x)+g(x)\hat{u}_i(x)\right)+Q(x)+r\hat{u}_i^{2}(x)=0
\end{eqnarray}
and denote
$\displaystyle\tilde{u}_{i+1}(x) = -\frac{1}{2r}g(x)^T\nabla\tilde{V}_{i}(x)^T$.
\begin{lemma}
For each $i\ge 0$, we have
$\lim\limits_{N_{1},N_{2}\rightarrow\infty}\hat{V}_i(x)=\tilde{V}_{i}(x)$,
$\lim\limits_{N_{1},N_{2}\rightarrow\infty}\hat{u}_{i+1}(x)=\tilde{u}_{i+1}(x)$,
$\forall x\in\Omega$.
\end{lemma}
\begin{proof} By definition
\begin{eqnarray}
\hspace{-1cm}&&\tilde{V}_{i}(x(t_{k+1}))-\tilde{V}_{i}(x(t_{k}))\nonumber\\
\hspace{-1cm}&=&-\int_{t_{k}}^{t_{k+1}}[Q(x)+r\hat{u}_{i}^{2}(x)+2r\tilde{u}_{i+1}(x)\hat{v}_i(x)]dt \label{eq:tildev}
\end{eqnarray}

Let $\tilde{c}_{i,j}$ and $\tilde{w}_{i,j}$ be the constant weights such that $\tilde{V}_{i}(x)=\sum_{j=1}^{\infty}\tilde{c}_{i,j}\phi_{j}(x)$ and $\tilde{u}_{i+1}(x)=\sum_{j=1}^{\infty}\tilde{w}_{i,j}\phi_{j}(x)$. Then, by \eqref{eq:onlineformular} and \eqref{eq:tildev}, we have
$e_{i,k}=\theta_{i,k}^{T}\bar{W}_i+\xi_{i,k}$,
where
\begin{eqnarray*}
\bar{W_{i}}\!&\!=\!&\!~~
\left[\begin{array}{cccccccc}
\!\tilde{c}_{i,1}\!&\!
\tilde{c}_{i,2}\!&\!
\cdots\!&\!
\tilde{c}_{i,N_{_{1}}}\!&\!
\tilde{w}_{i,1}\!&\!
\tilde{w}_{i,2}\!&\!
\cdots\!&\!
\tilde{w}_{i,N_{_{2}}}
\end{array}\right]^T\\
&\!&\!-
\left[\begin{array}{cccccccc}
\hat{c}_{i,1}\!&\!
\hat{c}_{i,2}\!&\!
\cdots\!&\!
\hat{c}_{i,N_{_{1}}}\!&\!
\hat{w}_{i,1}\!&\!
\hat{w}_{i,2}\!&\!
\cdots\!&\!
\hat{w}_{i,N_{_{2}}}
\end{array}\right]^T,\\
\xi_{i,k}\!\!&\!\!=\!&\!\sum_{j=N_{1}+1}^{\infty}\tilde{c}_{i,j}\left[\phi_{j}(x(t_{k+1}))-\phi_{j}(x(t_{k}))\right]\\
&&+\sum_{j=N_{2}+1}^{\infty}\tilde{w}_{i,j}\int_{t_{k}}^{t_{k+1}}2r\phi_{j}(x)\hat{v}_idt.
\end{eqnarray*}

Since the weights are found using the least-squares method, we have
\[\sum_{k=1}^l e_{i,k}^2\le\sum_{k=1}^l \xi_{i,k}^2\]

Also notice that,
\[
\sum\limits_{k=1}^l\bar{W}_{i}^{T}\theta_{i,k}^{T}\theta_{i,k}\bar{W_{i}}=\sum\limits_{k=1}^l(e_{i,k}-\xi_{i,k})^{2}
\]

Then, under Assumption \ref{as:PE}, it follows that
\[
\bar{|W_{i}|}^{2}\le\frac{4|\Xi_{i,l}|^{2}}{l\delta}=\frac{4}{\delta}\max_{1\le k\le l}\xi_{i,k}^{2}.
\]

Therefore, given any arbitrary $\epsilon>0$, we can find $N_{10}>0$ and $N_{20}>0$, such that when $N_1>N_{10}$ and $N_2>N_{20}$, we have
\begin{eqnarray}
&&|\hat{V}_i(x)-\tilde{V}_i(x)|\\
&\le& \sum_{j=1}^{N_1}|c_{i,j}-\hat c_{i,j}||\phi_j(x)|+\sum_{j=N_1+1}^{\infty}|c_{i,j}\phi_j(x)|\\
&\le& \frac{\epsilon}{2}+ \frac{\epsilon}{2}=\epsilon,~~\forall x\in\Omega.
\end{eqnarray}
Similarly, $|\hat{u}_{i+1}(x)-\tilde{u}_{i+1}(x)|\le \epsilon$. The proof is complete.
\end{proof}

We now prove Theorem 3.1 by induction:

\noindent 1) If $i=0$ we have $\tilde{V_{0}}(x)=V_{0}(x)$, and $\tilde{u}_{1}(x)=u_{1}(x)$. Hence, the convergence can directly be proved by Lemma A.1.

\noindent 2) Suppose for some $i>0$, we have
$\lim_{N_{1},N_{2}\rightarrow\infty}\hat{V}_{i-1}(x)={V}_{i-1}(x)$,
$\lim_{N_{1},N_{2}\rightarrow\infty}\hat{u}_{i}(x)={u}_{i}(x)$, $\forall x\in\Omega$. By definition, we have
\begin{eqnarray*}
&&|V_i(x(t))-\tilde{V}_i(x(t))|\\
&=&r|\int_{t}^{\infty}\left[\hat{u}_i(x)^2-u_i(x)^2\right]dt|\\
&+&2r|\int_{t}^{\infty}u_{i+1}(x)g(x)\left[\hat{u}_i(x)-u_i(x)\right]dt|\\
&+&2r|\int_{t}^{\infty}\left[\hat{u}_{i+1}(x)-u_{i+1}(x)\right]g(x)\hat{v}_idt|,~~\forall x\in\Omega.
\end{eqnarray*}

By the induction assumptions, we known
\begin{eqnarray}
&&\hspace{-1.5cm}\lim\limits_{N_1,N_2\rightarrow\infty}\int_{t}^{\infty}\left[\hat{u}_i(x)^2-u_i(x)^2\right]dt=0\\
&&\hspace{-1.5cm}\lim\limits_{N_1,N_2\rightarrow\infty}\int_{t}^{\infty}u_{i+1}(x)g(x)\left[\hat{u}_i(x)-u_i(x)\right]dt=0
\end{eqnarray}
Also, by Assumption \ref{as:PE}, we conclude
\begin{eqnarray}
\lim\limits_{N_1,N_2\rightarrow\infty}|{u}_{i+1}(x)-\hat{u}_{i+1}(x)|=0
\end{eqnarray}
and
\begin{eqnarray}
\lim\limits_{N_1,N_2\rightarrow\infty}|{V}_i(x)-\tilde{V}_i(x)|=0.
\end{eqnarray}

Finally, since
\begin{eqnarray*}
|\hat{V}_{i}(x)-V_{i}(x)|
\le|V_{i}(x)-\tilde{V}_{i}(x)|+|\tilde{V}_{i}(x)-\hat{V}_{i}(x)|
\end{eqnarray*}
and by the induction assumption, we have
\begin{eqnarray}
\lim\limits_{N_1,N_2\rightarrow\infty}|{V}_i(x)-\hat{V}_i(x)|=0.
\end{eqnarray}
The proof is thus complete.

\section*{Proof of Theorem 3.2}
Define
\begin{eqnarray}
\bar{e}_{ro}(x)=\left\{\begin{array}{cc}e_{ro}(x),&V_{i^*}(x)\le d\\0,&V_{i^*}(x)> d\end{array}\right.
\end{eqnarray}
and
\begin{eqnarray}
u(x)=u_{i^*}(x)+\frac{r}{2}\rho^{2}(|x|^2){u}_{i^*+1}(x)+\bar{e}_{ro}(x) \label{eq:generalized_u}
\end{eqnarray}
Then, along the solutions of (\ref{eq:mc2}), by completing the squares, we have
\begin{eqnarray*}
&&\dot{V}_{i^*}(x) \\
%
 %
 %
 %
 %
 %
 & \le & -Q(x)+\frac{1}{\rho^{2}(|x|^2)}(\Delta+\bar{e}_{ro}(x))^{2}\\
 %
 & = & -(Q(x)-\epsilon^{2}|x|^{2})-\frac{4\gamma^2-(\Delta+\bar{e}_{ro}(x))^{2}}{\rho^{2}(|x|^2)}\\
 & \le & -Q_0(x)-4\frac{\gamma^{2}-\max\{\kappa^{2}_1(|w|),\kappa^{2}_2(|x|),\bar{e}_{ro}^2(|x|)\}}{\rho^{2}(|x|^2)}
\end{eqnarray*}
where $Q_0(x)=Q(x)-\epsilon^2|x|^2$ is a positive definite function of $x$.

Therefore, under Assumptions \ref{as:HJB}, \ref{as:duc} and the gain condition \eqref{eq:sg_cd}, we have the following implication:
\begin{eqnarray}
&&\!\!V_{i^*}(x)
\ge
{\bar{\alpha}}\!\circ\!\gamma^{-1}\!\circ\! \kappa_1\!\circ\!{\underline{\lambda}}^{-1}\!\left(W(w)\right)\nonumber\\
\!\!&\!\!\Rightarrow\!\!&\!\!
|x|\ge\gamma^{-1}\circ \kappa_1\circ{\underline{\lambda}}^{-1}\left(W(w)\right)\nonumber\\
\!\!&\!\!\Rightarrow\!\!&\!\!
\gamma\left(|x|\right)\ge\kappa_1\left(|w|\right)\\
\!\!&\!\!\Rightarrow\!\!&\!\!
\gamma\left(|x|\right)\ge\max\{\kappa_1\left(|w|\right),\kappa_2\left(|x|\right),\bar{e}_{ro}\left(|x|\right)\}\nonumber\\
\!\!&\!\!\Rightarrow\!\!&\!\!
\dot V_{i^*}(x)\le - Q_0(x).\nonumber
\end{eqnarray}

Also, under Assumption \ref{as:duc}, we have
\begin{eqnarray}
&&W(w)\ge \kappa_3\circ{{\underline\alpha}^{-1}(V_{i^*}(x) )}\nonumber\\
&\Rightarrow& W(w)\ge \kappa_3 (|x|)\nonumber\\
&\Rightarrow& \nabla W(w)\Delta_w(w,x)\le-\kappa_4(|w|).
\end{eqnarray}

Finally, under the gain condition \eqref{eq:sg_cd}, it follows that
\begin{eqnarray}
&&{\gamma}(s)>\kappa_1\circ{\underline{\lambda}}^{-1}\circ\kappa_3\circ{\underline{\alpha}}^{-1}\circ{\bar{\alpha}}(s)             \nonumber\\
&\Rightarrow&
{\gamma}\circ{\bar{\alpha}}^{-1}(s')>\kappa_1\circ{\underline{\lambda}}^{-1}\circ\kappa_3\circ{\underline{\alpha}}^{-1}(s')          \\
&\Rightarrow&
s'>{\bar{\alpha}}\circ\gamma^{-1}\circ\kappa_1\circ{\underline{\lambda}}^{-1}\circ\kappa_3\circ{\underline{\alpha}}^{-1}(s')       \nonumber
\end{eqnarray}
where $s'=\bar\alpha(s)$. Hence, the following small-gain condition holds:
\begin{eqnarray}
\left[{\bar{\alpha}}\circ\gamma^{-1}\circ\kappa_1\circ{\underline{\lambda}}^{-1}\right]\circ\left[\kappa_3\circ{\underline{\alpha}}^{-1}(s)\right]<s,~~\forall s>0. \end{eqnarray}

By Theorem 3.1 in \cite{jiang1996a}, the system (\ref{eq:mc1}), (\ref{eq:mc2}), \eqref{eq:generalized_u} is globally asymptotically stable at the origin.


Next, denote $\chi_1={\bar{\alpha}}\circ\gamma^{-1}\circ\kappa_1\circ{\underline{\lambda}}^{-1}$, and $\chi_2=\kappa_3\circ{\underline{\alpha}}^{-1}$.
Also, let $\hat{\chi}_1$ be a function of class $\mathcal{K}_{\infty}$ such that
\begin{enumerate}
\item $\hat{\chi}_1(s)\le\chi_1^{-1}(s)$, $\forall s\in[0,\lim\limits_{s\rightarrow\infty}\chi_1(s))$,
\item ${\chi}_2(s)\le\hat{\chi}_1(s)$, $\forall s\ge 0$.
\end{enumerate}

Then, as shown in \cite{jiang1996a}, there exists a continuously differentiable class $\mathcal{K}_{\infty}$ function $\sigma(s)$ satisfying $\sigma'(s)>0$ and $\chi_2(s)<\sigma(s)<\hat{\chi}_1(s)$, $\forall s>0$, such that
the set
\begin{eqnarray}
\Omega_{i^*}=\{(w,x): \max\left[\sigma(V_{i^*}(x)),W(w)\right]\le d\}
\end{eqnarray}
is an estimate of the region of attraction of the closed-loop system composed of \eqref{eq:mc1}, \eqref{eq:mc2}, and \eqref{eq:acon}.

The proof is thus complete.

\section*{Proof of Theorem 4.2}
Define
\begin{eqnarray*}
\bar{e}_{ro1}(X_1)&=&\left\{\begin{array}{cc}e_{ro1}(X_1),&U_{i^*}(X_1)\le d_1,\\0,&U_{i^*}(X_1)> d_1,\end{array}\right.\\
\bar{\bar{e}}_{ro}(x)&=&\left\{\begin{array}{cc}e_{ro}(x),&U_{i^*}(X_1)\le d_1,\\0,&U_{i^*}(X_1)> d_1,\end{array}\right.
\end{eqnarray*}
Along the solutions of \eqref{eq:muc1}-\eqref{eq:muc3} with the control policy
\begin{eqnarray*}
\!\!{u}
\!\!&\!\!=\!\!&\!\!-\bar{f}_1(x,z)
+2r\hat{u}_{i^*+1}(x)-\frac{\bar{g}^2(x)\rho_1^2(|X_1|^2)\zeta}{4}\\
&&-\frac{\rho_{1}^2(|X_1|^2)\zeta}{4}-\frac{\epsilon^2\rho^2(\zeta^2)\zeta}{2\rho^2(|x|^2)}-\epsilon^2\zeta-\bar{e}_{ro1}(X_1), \nonumber
\end{eqnarray*}
it follows that
\begin{eqnarray*}
\dot U_{i^*}(X_1)
\!\!&\!\le\!&\! -Q_0(x)-\frac{1}{2}\epsilon^2\zeta^2\\
\!\!&\!\!&\!-\frac{\gamma^2_1(|X_1|)\!-\!\max\{\tilde\kappa_1^2(|w|),\tilde\kappa_2^2(|X_1|),\bar{\bar{e}}_{ro}^2(x)\}}{\frac{1}{4}\rho^2(|x|^2)}\\
\!\!&\!\!&\!-\frac{\gamma^2_1(|X_1|)\!-\!\max\{\tilde\kappa_1^2(|w|),\tilde\kappa_2^2(|X_1|),\bar{\bar{e}}_{ro}^2(x)\}}{\frac{1}{4}\rho^2_{1}(|X_1|^2)}\\
\!\!&\!\!&\!-\frac{\gamma^2_1(|X_1|)\!-\!\max\{\tilde\kappa_1^2(|w|),\tilde\kappa_2^2(|X_1|),\bar{e}_{ro1}^2(X_1)\}}{\frac{1}{4}\rho^2_{1}(|X_1|^2)}
\end{eqnarray*}
As a result,
\begin{eqnarray*}
&U_{i^*}(X_1)\le\max\{\bar{\alpha}_1\!\circ\!\gamma^{-1}_1\!\circ\!\tilde\kappa_1\!\circ\!\underline\lambda^{-1}(W(w)),
\bar{\alpha}_1\!\circ\!\gamma^{-1}_1(|v|)\}&\nonumber\\
&\Rightarrow
 \dot U_{i^*}(X_1)\le -Q_0(x)+\frac{1}{2}\epsilon^2|\zeta|^2&.
\end{eqnarray*}

The rest of the proof follows the same reasoning as in the proof of Theorem 3.2.


\end{document}